\documentclass{amsart}

\usepackage{amssymb}
\usepackage{amsthm}  
\usepackage{amsmath} 
\usepackage{amscd} 
\usepackage[all]{xypic}
\usepackage{url}

\newcommand{\omi}{\omega^{-1}}
\newcommand{\om}{\omega}
\newcommand{\ka}{\kappa}
\newcommand{\si}{\sigma}

\newcommand{\ba}{\mathcal{G}}  
\newcommand{\va}{\varphi}
\newcommand{\fg}{\mathfrak g}
\newcommand{\fp}{\mathfrak p}
\newcommand{\Ad}{{\rm Ad}}
\newcommand{\Fl}{{\rm Fl}}
\newcommand{\na}{\nabla}
\newcommand{\U}{\Upsilon}
\newcommand{\gr}{{\rm gr}}
\newcommand{\lz}{[\![}
\newcommand{\pz}{]\!]}
\newcommand{\R}{\mathbb{R}}

\newtheorem{prop}{Proposition}[section]
\newtheorem*{prop*}{Proposition}
\newtheorem{thm}[prop]{Theorem}
\newtheorem*{thm*}{Theorem}
\newtheorem{lemma}[prop]{Lemma}
\newtheorem*{lemma*}{Lemma}
\newtheorem{cor}[prop]{Corollary}
\newtheorem*{cor*}{Corollary}
\newtheorem*{rem*}{Remark}
\theoremstyle{definition}
\newtheorem{deff}[prop]{Definition}
\newtheorem*{deff*}{Definition}
\theoremstyle{remark}
\newtheorem{exam}[prop]{Example}
\newtheorem*{exam*}{Example}
\newtheorem{rem}[prop]{Remark}

\begin{document}
\title{Symmetries of parabolic geometries}
\author{Lenka Zalabov\'a}
\address{
Masaryk University, Brno, Czech Republic and 
International Erwin Schr\" odinger Institute for Mathematical Physics, Vienna, Austria}
\email{zalabova@math.muni.cz}
\begin{abstract} 
We generalize the concept of affine locally symmetric spaces for parabolic geometries. 
We discuss mainly $|1|$--graded geometries and we show some restrictions on their 
curvature coming from the existence of symmetries. 
We use the theory of Weyl structures to discuss more interesting $|1|$--graded 
geometries which can carry a symmetry in a point with nonzero curvature.
More concretely, we discuss the number of different symmetries which can exist at the 
point with nonzero curvature.
\end{abstract}

\thanks{Discussions with J. Slov\'ak, V. \v Z\'adn\' ik and A. \v Cap were very useful 
during the work on this paper.  This research has been supported at different times 
by the grant GACR 201/05/H005 and by  ESI Junior Fellows program}

\maketitle

We introduce and discuss symmetries for the so called parabolic geometries.  
Our motivation comes from locally symmetric spaces. 
Remind that affine locally symmetric space is a manifold  with locally defined 
symmetries at each point together with an affine connection which is invariant 
with respect to the symmetries. It can be proved that the manifold with 
an affine connection is locally symmetric if and only if the torsion vanishes 
and the curvature is covariantly constant. See \cite{KN2,H} for detailed discussion 
of affine locally symmetric spaces.

We are interested in $|1|$--graded parabolic geometries. In this case, 
the definition of the symmetry is a generalization of the classical one 
and follows the intuitive idea. We show an analogy of the facts known for 
the affine locally symmetric spaces and we give further results which 
come from the general theory of parabolic geometries. 
See also \cite{ja} for discussion of torsion restrictions for symmetric 
$|1|$--graded geometries.

Among all symmetric $|1|$--graded geometries, there are some `interesting' types of 
them, which can carry a symmetry in the point and still allow some nonzero curvature 
at this point. 
We use the theory of Weyl structures to discuss these geometries, see \cite{WS}. 
In particular, we study the action of symmetries on Weyl structures and 
via these actions, we describe the Weyl curvature of the geometries.
Finally we discuss the question on the number of different symmetries which can exist in the 
point with nonzero curvature.

\section{Introduction to Symmetries}
The aim of this section is to introduce the definition and discuss basic properties 
of symmetries for parabolic geometries. We first remind some basic definitions 
and facts on Cartan and parabolic geometries.
In this paper we follow the concepts and notation of \cite{WS, parabook} and the reader
is advised to consult \cite{TC}.

\subsection*{Cartan geometries}
Let $G$ be a Lie group, $P\subset G$ a Lie subgroup, and write $\fg$ and
$\fp$ for their Lie algebras. A \emph{Cartan geometry} of type $(G,P)$ on a
smooth manifold $M$ is a principal  $P$--bundle $p:\ba \rightarrow M$ together with the
$1$--form $\om \in \Omega^1(\ba,\fg)$ called
a \emph{Cartan connection} such that:
\begin{enumerate}
\item $(r^h)^*\om=\Ad_{h^{-1}} \circ \om$ for each $h \in P$,
\item $\om(\zeta_X(u))=X$ for each fundamental vector field $\zeta_X$, $X \in \fp$, 
\item $\om(u):T_u\ba \longrightarrow  \fg$ is a linear isomorphism for 
each $u \in \ba$.
\end{enumerate} 
The simplest examples are so called \emph{homogeneous models}, which are the 
$P$--bundles $G\rightarrow G/P$ endowed with the (left) Maurer Cartan form $\om_G$. 

The absolute parallelism $\om$ provides the existence of \emph{constant vector fields} 
$\omi(X)$ from $\frak X(\ba)$ defined  for all $X \in \fg$ by $\om (\omi(X)(u))=X$ for 
all $u \in \ba$. 
The flow lines of constant vector field $\omi(X)$ we denote by $\Fl^{\omi(X)}_t(u)$
for $u \in \ba$.

A \emph{morphism of Cartan geometries} of the same type from 
$(\ba \rightarrow M,\om)$ to $(\ba'\rightarrow M',\om')$ is a principal bundle 
morphism $\va:\ba \rightarrow \ba'$ such that $\va^*\om'=\om$. Further we  
denote the base morphism of $\va$ by $\underline \va:M \rightarrow M'$.
Remark that each morphism of Cartan geometries preserves
constant vector fields and hence preserves flows of constant vector 
fields, thus
\begin{align*}
T\va \circ \omi(X)&= \om'^{-1}(X) \circ \va, 
\\ \va \circ \text{\Fl}^{\omi(X)}_t(u)&=\text{\Fl}^{\om'^{-1}(X)}_t(\va(u))
\end{align*}
hold for all  $X \in \fg$. 

We define the \emph{kernel} of the geometry of type $(G,P)$ 
as the maximal normal subgroup of $G$ which is contained in $P$. 
The geometry is called 
\emph{effective} if the kernel is trivial and the geometry is called 
\emph{infinitesimally effective} if the kernel is discrete. 
The following Theorem describes useful properties of morphisms, see \cite{S}:
\begin{thm} \label{rigidity}
Let $(\ba \rightarrow M,\om)$ and $(\ba' \rightarrow M',\om')$ be Cartan 
geometries of type $(G,P)$ and denote by $K$ the kernel. Let $\va_1$ 
and $\va_2$ be morphisms of these Cartan geometries which cover the 
same base morphism $\underline \va:M\rightarrow M'$.
Then there exists a smooth function $f:\ba \rightarrow K$ such that 
$\va_1(u)=\va_2(u) \cdot f(u)$ for all $u \in \ba$. 

In particular, if the geometry is effective, then $\va_1=\va_2$ and $f$ is 
constant on connected components of $M$ for infinitesimally effective geometries.
\end{thm}

We shall mainly deal with automorphisms of Cartan geometries. These are 
$P$--bundle automorphisms, which preserve the Cartan connection of the geometry. 
In the homogeneous case, there
is the Liouville theorem, see \cite{S}:
\begin{thm} \label{auto} Let $G/P$ be connected.
All (locally defined) automorphisms of the homogeneous model 
$(G\rightarrow G/P,\om_G)$ are  left multiplications by elements of $G$.
\end{thm}

The structure equation defines the horizontal smooth $2$--form 
$K\in \Omega^2(\ba,\fg)$ in the following way:
$$
K(\xi,\eta)=d\om(\xi,\eta)+[\om(\xi),\om(\eta)]
.$$ 
This makes sense for each Cartan connection $\om$ and the form $K$ is called 
the \emph{curvature form}. 
Notice that the Maurer--Cartan equation implies that the curvature of 
homogeneous model is zero. 
It can be proved, see \cite{S}: 
\begin{thm}\label{thm-locally-flat}
If the curvature of a Cartan geometry of type $(G,P)$ vanishes, then the geometry 
is locally isomorphic to the homogeneous model $(G \rightarrow G/P,\om_G)$.
\end{thm}
If the curvature vanishes, then the Cartan geometry is called \emph{locally flat}. 
Homogeneous models are sometimes called \emph{flat} models.

The curvature can be equivalently described by means of the parallelism 
by the \emph{curvature function}
$\ka:\ba \rightarrow \wedge^2 (\fg/\fp)^* \otimes \fg$, where 
\begin{align*}
\ka(u)(X,Y)&=K(\om^{-1}(X)(u),\om^{-1}(Y)(u))=\\
&=[X,Y]-\om([\om^{-1}(X),\om^{-1}(Y)](u))
.
\end{align*}
If the values of $\ka$ are in $\wedge^2(\fg/\fp)^* \otimes
\fp$, we call the geometry \emph{torsion free}.
We will use both descriptions of the curvature and we will not distinguish 
between them.

\subsection*{Parabolic geometries}
Let $\fg$ be a semisimple Lie algebra.
A \emph{$|k|$--grading on $\fg$} is a vector space decomposition 
$$
\fg=\fg_{-k}\oplus \dots   \oplus \fg_{0} \oplus \dots \oplus\fg_{k}
$$ 
such that $[\fg_i,\fg_j]\subset \fg_{i+j}$ for all $i$ and $j$ (we
understand $\fg_{r}=0$
for $|r|>k$) and such that the subalgebra $\fg_{-}:=\fg_{-k}\oplus \dots
\oplus \fg_{-1}$ is generated by $\fg_{-1}$. We suppose that there is
no simple ideal of $\fg$ contained in $\fg_0$ and that the grading on $\fg$ is fixed. 
Each grading of $\fg$ defines the \emph{filtration} 
$$
\fg=\fg^{-k}\supset \fg^{-k+1} \supset \dots
\supset \fg^{k}=\fg_k
,$$ 
where $\fg^{i}=\fg_i \oplus \dots \oplus \fg_k$. In
particular, $\fg_0$ and $\fg^0=:\fp$ are subalgebras of $\fg$ and
$\fg^1=:\fp_{+}$ is a nilpotent ideal in $\fp$.

Let $G$ be a semisimple Lie group with the Lie algebra $\fg$. 
To get a geometry we have to choose Lie subgroups
$G_0\subset P\subset G$ with prescribed subalgebras 
$\fg_0$ and $\fp$.
The obvious choice is this one: 
\begin{align*}
G_0 &:=\{g \in G \mid \Ad_g(\fg_i)\subset \fg_i, 
\ \forall i=-k,\dots,k \},
\\ 
P &:=\{g \in G \mid \Ad_g(\fg^i)\subset \fg^i,\  \forall i=-k,\dots,k\}
.\end{align*}
This is the maximal possible choice,
but we may also take the connected component of the unit in these subgroups
or anything between these two extremes.  
It is not difficult to show for these subgroups, see \cite{Y}:
\begin{prop} \label{rozklad}
Let $\fg$ be a $|k|$--graded semisimple Lie algebra and $G$ be a Lie group 
with Lie algebra $\fg$. \\
(1) $G_0 \subset P \subset G$ are closed subgroups with Lie algebras 
$\fg_0$ and $\fp$, respectively. \\
(2) The map $(g_0,Z)\mapsto g_0\exp Z$ defines a diffeomorphism 
$G_0 \times \fp_+ \rightarrow P$.
\end{prop}
The group $P$ is a semidirect product of the reductive subgroup 
$G_0$ and the nilpotent normal subgroup $P_+:=\exp \fp_+$ of $P$.

A \emph{parabolic geometry} is a Cartan geometry of type $(G,P)$, where $G$
and $P$ are as above. If the length of the grading of $\fg$ is $k$, then
the geometry is called \emph{$|k|$--graded}.
Parabolic geometries are infinitesimally effective, but they 
are not effective in general. 

\begin{exam}\emph{Conformal Riemannian structures.} \label{exam-conf}
We take the Cartan geometry of type $(G,P)$ where $G=O(p+1,q+1)$ is the 
orthogonal group and $P$ is the Poincar\'e conformal subgroup. 
Thus the group $G$ is of the form
$$
G=\left\{A \ \Big|\ A 
\left( \begin{smallmatrix}
0&0&1\\0&J&0\\1&0&0 \end{smallmatrix} \right) A^{-1} = 
\left( \begin{smallmatrix}
0&0&1\\0&J&0\\1&0&0 \end{smallmatrix} \right) \right\}
,$$ 
where $J=\left( \begin{smallmatrix}
E_p&0\\0&-E_q \end{smallmatrix} \right)$ 
is the standard product of signature 
$(p,q)$. Its Lie algebra $\fg=\frak o(p+1,q+1)$ can be written as
$$
\fg=\left\{ \left( \begin{smallmatrix}
a&Z&0\\X&A&-JZ^T\\0&-X^TJ&-a \end{smallmatrix} \right)\ \bigg|\ a 
\in \mathbb{R},\ X,Z^T \in\mathbb{R}^{p+q},\ 
A \in \frak o(p,q) \right\}
.$$ 
It leads to a decomposition
$\fg=\fg_{-1} \oplus \fg_0 \oplus \fg_1$, where $\fg_{-1}\simeq \mathbb{R}^{p+q}$, 
$\fg_{0}\simeq \mathbb{R} \oplus \frak o(p,q)$ and $\fg_{1}\simeq \mathbb{R}^{p+q*}$. 
These parts correspond to the block lower triangular part, block diagonal 
part and block upper triangular part and give exactly the $|1|$--grading 
of $\fg$.
The elements of the subgroup 
$$
G_0=\left\{ \left( \begin{smallmatrix}
\lambda&0&0\\0&C&0\\0&0&\lambda^{-1} \end{smallmatrix} \right)
 \ \Big|\ \lambda \in \mathbb{R}\setminus \{0\},\ C \in O(p,q)\right\}
$$
preserve this grading. Elements of the subgroup $P=G_0 \rtimes \exp \fg_{1}$ 
preserving the filtration are of the form
$$
\bigg( \begin{smallmatrix}
\lambda&0&0\\0&C&0\\0&0&\lambda^{-1} \end{smallmatrix} \bigg)\cdot 
\bigg( \begin{smallmatrix}
1&Z& -{1 \over 2}ZJZ^T\\0&E&-JZ^T\\0&0&1 \end{smallmatrix} \bigg)=
\bigg( \begin{smallmatrix}
\lambda &\lambda Z& -{\lambda \over 2}ZJZ^T\\0&C&-CJZ^T\\0&0&\lambda^{-1} 
\end{smallmatrix} \bigg)
.$$ 
We get the $|1|$--graded geometry and its homogeneous model is the conformal 
pseudosphere of the corresponding signature.
\end{exam}

We are mainly interested in the curvature of parabolic geometries
and we discuss here the curvature in more detail.
It has its values in the cochains $C^2(\fg_{-},\fg)$
for the second cohomology $H^2(\fg_{-},\fg)$. This group can be also computed 
as the homology of the codifferential
$\partial^* : \wedge ^{k+1}\fg_{-}^*\otimes\fg\rightarrow
\wedge^{k}\fg_{-}^*\otimes\fg$. On the decomposable elements it is given by
\begin{align*}
\partial^*(Z_0 \wedge \dots \wedge
Z_k \otimes W) =\ &\sum_{i=0}^k(-1)^{i+1} Z_0 
\wedge \cdots \hat i \cdots \wedge
Z_k \otimes [Z_i,W]  + 
\\&\sum_{i<j}(-1)^{i+j}[Z_i,Z_j]\wedge \cdots
\hat i \cdots \hat j \cdots \wedge Z_k \otimes W
\end{align*}
for all $Z_0,\dots, Z_k
\in \fp_+$ and $W \in \fg$, where the hats denote ommission. 
We use here the pairing between $\fg_{-}$ and $\fp_+$ given by the 
Killing form on $\fg$. 

The parabolic geometry is called \emph{normal} if the curvature satisfies
$$
\partial ^* \circ \ka =0
.$$  
If the geometry is normal, we can define the
\emph{harmonic part of curvature} $\ka_H:\ba \rightarrow H^2(\fg_{-},\fg)$
as the composition of the curvature function and the projection to the second
cohomology group.

Thanks to the grading of $\fg$, there are several decompositions of the
curvature of the parabolic geometry. One of the possibilities is the
\emph{decomposition into homogeneous components}, which is of the form
$$
\ka=\sum_{i=-k+2}^{3k}\ka^{(i)}
,$$ 
where $\ka^{(i)}(u)(X,Y) \in \fg_{p+q+i}$
for all $X \in \fg_p,Y\in \fg_q$ and $u \in \ba$. 

The parabolic geometry is
called \emph{regular} if the curvature function $\ka$ satisfies $\ka^{(r)}=0$
for all $r\leq 0$.  The crucial structural description of the curvature is
provided by the following Theorem, see \cite{Y}:

\begin{thm}\label{thm-harmonic-curvature}
The curvature $\ka$ of a regular normal parabolic geometry vanishes if and only 
if its harmonic part $\ka_{H}$ vanishes. 
Moreover, if all homogeneous components of
$\ka$ of degrees less than $j$ vanish identically and there is no
cohomology $H^2_j(\fg_{-},\fg)$, then also the curvature component of degree
$j$ vanishes.
\end{thm} 

Another possibility is the decomposition of the curvature according to the
values:
$$
\ka=\sum_{j=-k}^{k}\ka_j
.$$ 
In an arbitrary frame $u$ we
have $\ka_j(u) \in \fg_{-} \wedge \fg_{-} \rightarrow \fg_j$. 
The component
$\ka_{-}$ valued in $\fg_{-}$ is exactly the torsion of the geometry.

Remark that in the 
$|1|$--graded case
the decomposition by the homogeneity
corresponds to the decomposition according to the values. The homogeneous
component of degree $1$ coincides with the torsion while the homogeneous
components of degrees $2$ and $3$ correspond to $\ka_0$ and $\ka_1$. 
Thus all $|1|$--graded geometries are clearly regular.

\subsection*{Underlying structures for parabolic geometries and tractor bundles} \label{tract}
It is well known, that the existence of the Cartan connection allows to describe 
the tangent bundle of the base manifold as the associated bundle
$TM \simeq  \ba \times_P \fg /\fp$, where the action of $P$ on $\fg/\fp$ 
is induced by the $\Ad$--action on $\fg$. 
In the case of parabolic geometry, we can use the usual identification
$\fg/\fp \simeq \fg_{-}$. Thus we can write $TM \simeq \ba \times_P \fg_{-}$ 
for the so called $\underline{\Ad}$--action of $P$ on $\fg_{-}$, which is  obviously 
given e.g. by the condition $\underline{\Ad} = \pi \circ \Ad$, where $\pi$ is the 
projection  $\fg \rightarrow \fg_{-}$ in the direction of $\fp$. 
In particular, if we denote 
$p:\ba \rightarrow M$ the projection, then one can see from the properties 
of Cartan connection that 
the mapping $\ba \times \fg_{-} \rightarrow TM$
given by $(u,X) \mapsto Tp.\omi(X)(u)$ factorizes to the bundle isomorphism 
$TM \simeq \ba \times_P \fg_{-}$. We use the notation
$\lz u,X \pz$ for the tangent vectors and more generally, we denote elements 
and sections of arbitrary asociated bundle by the same symbol.

Remark that vector fields can be equivalently understand as
$P$--equivariant mappings $\ba \rightarrow \fg_{-}$, 
so called \emph{frame forms}. Each frame form $s$ describes 
the vector $\lz u, s(u) \pz \in T_{p(u)}M$.
We have similar identifications for the cotangent bundle and arbitrary 
tensor bundles. 

The morphism $\va$ of Cartan geometries induces uniquely by means of 
its base morphism $\underline \va$ the tangent morphism 
$T\underline \va:TM \rightarrow TM'$ and it can be written using the 
previous identification as
$$
T \underline \va (\lz u, X \pz) 
= \lz \va(u),X \pz. 
$$
Again, similar fact works for the cotangent bundle and arbitrary 
tensor bundles and can be rewritten in the language of frame forms. In the 
sequel, we will use both descriptions of tensors and tensor fields.

There is also a different concept of interesting associated bundles available 
for all Cartan geometries. 
First, we can define the \emph{adjoint tractor bundle} $\mathcal{A}M$ as the 
associated bundle $\mathcal{A}M := \ba \times_P \fg$ with respect to the 
adjoint action. The Lie bracket on $\fg$ defines a bundle map 
$$
\{\ ,\ \}:\mathcal{A}M \otimes \mathcal{A}M \rightarrow  \mathcal{A}M
$$ 
which makes any fiber of $\mathcal{A}M$ into a Lie algebra isomorphic to the 
Lie algebra $\fg$. For all $u \in \ba$ and $X,Y \in \fg$
the bracket is defined by 
$$ 
\{ \lz u, X \pz , \lz u,Y \pz \}=\lz u, [X,Y] \pz
.$$

More generally, let $\lambda: G \rightarrow Gl(V)$ be a linear representation. 
We define the \emph{tractor bundle} $\mathcal{V}M$ as the associated bundle 
$\mathcal{V}M:=\ba \times_P V$ with respect to the restriction of the action 
$\lambda$ to the subgroup $P$. In the case of the adjoint representation 
$\Ad: G \rightarrow Gl(\fg)$ we get exactly the adjoint bundle. 
Elements of the associated bundle and also the sections of the bundle are 
called (\emph{adjoint}) \emph{tractors}. 

These bundles allow us to describe nicely structures underlying parabolic 
geometries and descriptions via them will be useful for 
some computations later. 
We will not devote the whole theory. We will only shortly remind basic facts 
which allow us to use tractor language in the future.
Note that one has to start with the representation of the whole $G$ to 
define tractor bundles. There are differences between properties of tractor bundles 
and e.g. tangent bundles. See \cite{TBI, TC} for detailed discussion.

The adjoint tractor bundle $\mathcal{A}M $ acts on an arbitrary tractor bundle
$\mathcal{V}M=\ba \times_P V$. For all $u \in \ba$, $X \in \fg$ and $v \in V$
we define the \emph{algebraic action} 
$$
\bullet : \mathcal{A}M \otimes \mathcal{V}M \rightarrow \mathcal{V}M,
\ \ \ \ 
\lz u,X\pz \bullet \lz u, v \pz=\lz u, \lambda'_X(v) \pz
,$$
where  $\lambda':\fg \rightarrow \frak{gl}(V)$ is the infinitesimal 
representation given by $\lambda: G \rightarrow Gl(V)$. 

In fact, any $G$--representation $\lambda$ gives rise to the natural bundle 
$\mathcal{V}$ on Cartan geometries of type $(G,P)$ and all $P$--invariant 
operations on representations give rise to geometric operations on the 
corresponding natural bundles. Similarly, the $P$--equivariant morphisms 
of representations give corresponding bundle morphisms.

In this way, the projection $\pi:\fg \rightarrow \fg/\fp$ naturally 
induces the bundle projection 
$$
\Pi:\ba \times_P \fg = \mathcal{A}M \longrightarrow TM=\ba \times_P \fg/\fp
$$ 
and we can see the adjoint tractor bundle as an extension of the tangent bundle.

Now, suppose that we have a $|k|$--graded parabolic geometry. The filtration 
of $\fg$ is $P$--invariant and induces the filtration of the adjoint subbundles 
$$ 
\mathcal{A}M =\mathcal{A}^{-k}M \supset \mathcal{A}^{-k+1}M \supset \dots \supset 
\mathcal{A}^{k}M
,$$ 
where $\mathcal{A}^iM:=\ba \times_P \fg^i$.
At the same time, we get the \emph{associated graded bundle} 
$$
\gr(\mathcal{A}M) = 
\mathcal{G} \times_P \gr(\fg) = 
\mathcal{A}_{-k}M \oplus \mathcal{A}_{-k+1}M \oplus 
\dots \oplus \mathcal{A}_{k}M
,$$
 where $\mathcal{A}_{i}M = \mathcal{A}^{i}M /\mathcal{A}^{i+1}M \simeq 
\ba \times_P \fg^{i}/\fg^{i+1}$.
Since the Lie bracket on $\gr(\fg)$ is $P_+$--invariant, 
there is  the \emph{algebraic bracket} on $\gr(\mathcal{A}M)$ defined by means 
of the Lie bracket. This bracket is compatible with the latter bracket 
on the adjoint tractor bundle and we denote both brackets by the same symbol. 
We have 
$$
\{\ ,\ \}: \mathcal{A}_iM \times \mathcal{A}_jM \rightarrow \mathcal{A}_{i+j}M
.$$  

From above we see that $TM \simeq \mathcal{A}M/\mathcal{A}^0M$ and we obtain the induced 
filtration of the tangent bundle 
$$
TM=T^{-k}M \supset T^{-k+1}M \supset \dots \supset T^{-1}M
,$$ 
where $T^iM \simeq \mathcal{A}^{i}M/\mathcal{A}^0M$. 
Again, the filtration of $TM$ gives rise to the associated graded bundle 
$$
\gr(TM)=\gr_{-k}(TM) \oplus \dots \oplus \gr_{-1}(TM)
,$$
 where $\gr_i(TM)=T^iM/T^{i+1}M \simeq \mathcal{A}_iM.$ 
The action of $P_+$ on $\ba$ is free and the quotient 
$\ba/P_+=:\ba_0\rightarrow M$ is a principal bundle with structure group 
$G_0=P/P_+$. 
The action of $P_+$ on $\fg^i/\fg^{i+1}$ is trivial and we have 
$\fg^i/\fg^{i+1} \simeq \fg_i$ as $G_0$--modules. 
We get $\gr_i(TM)\simeq \ba_0 \times_{G_0} \fg_i$ and 
$\gr(TM) \simeq \ba_0 \times_{G_0} \fg_-$. For each $x \in M$, 
the space $\gr(T_xM)$ is the nilpotent graded Lie algebra isomorphic 
to the algebra  $\fg_-$.

Now, we would like to describe the structure underlying parabolic geometries.
These are nicely related to filtered manifolds.
Remark, that a \emph{filtered manifold} is a manifold $M$ together
with a filtration $TM=T^{-k}M \supset \dots \supset T^{-1}M$ such that for
sections $\xi$ of $T^iM$ and $\eta$ of $T^jM$ the Lie bracket $[\xi,\eta]$
is a section of $T^{i+j}M$. 
On the corresponding associated graded bundle we obtain the \emph{Levi bracket} 
$$
\mathcal{L}:\gr(TM) \times \gr(TM) \rightarrow \gr(TM)
,$$ 
which is induced by 
$T^iM \times T^jM \rightarrow \gr_{i+j}(TM)$, 
the composition of the Lie bracket of the latter vector fields with the natural 
projection $T^{i+j}M \rightarrow  \gr_{i+j}(TM)$. 
This depends only on the classes in $\gr_i(TM)$ and $\gr_j(TM)$
and gives a map
$$
\gr_{i}(TM) \times \gr_{j}(TM) \rightarrow \gr_{i+j}(TM)
.$$ 
For each $x \in M$, this makes $\gr(T_xM)$ into a nilpotent graded Lie 
algebra.

For each parabolic geometry, we have described some canonical filtration of $TM$ 
induced from the grading of $\fg$.
It can be proved that geometry is regular if and only if this induced filtration 
of $TM$ makes $M$ into the filtered manifold such that the latter bracket on each 
$\gr(T_xM)$ coincides with $\{\ ,\ \}$.
In particular, we have $\gr(T_xM)\simeq \fg_{-}$ for each $x \in M$, 
see \cite{parabook}.

If we start with a regular parabolic geometry, we get exactly the following data 
on the base manifold:
\begin{itemize}
\item A filtration $\{T^iM\}$ of the tangent bundle such that $\gr(T_xM)\simeq \fg_{-}$
      for each $x \in M$.
\item A reduction of structure group of the associated graded bundle 
      $\gr(TM)$ with respect to $\Ad:G_0 \rightarrow \text{Aut}_\gr(\fg_-)$. 
      (The reduction is trivial in the case $G_{0} = \text{Aut}_\gr(\fg_-)$.)
\end{itemize}
These data are called the underlying \emph{infinitesimal flag structure}.
The proof of the following equivalence between
such infinitesimal flag structures and regular normal parabolic geometries 
can be found in \cite{Y, CS}:

\begin{thm}\label{thm-prolongation} 
Let $M$ be a filtered manifold such that $\gr(T_xM)\simeq \fg_{-}$ for each 
$x \in M$ and let $\ba_0 \rightarrow M$ be a reduction 
of $\gr(TM)$ to the structure group $G_0$. Then there is a regular normal
parabolic geometry $(p:\ba \rightarrow M,\om)$ inducing the given data. 
If $H^1_\ell(\fg_{-},\fg)$ are trivial for all $\ell > 0$ then the normal 
regular geometry is unique up to isomorphism.
\end{thm}
The construction is functorial and the latter Theorem describes an
equivalence of categories.

In the case of $|1|$--graded geometries, the filtration of the tangent bundle 
is trivial. We need only the reduction of $\gr(TM)$ to the structure group $G_0$. 
The $|1|$--graded geometries are automatically regular and we get the
correspondence between normal $|1|$--graded parabolic 
geometries and first order G--structures with structure group~$G_0$.

\subsection*{Symmetries and homogeneous models}
In general, automorphisms of Cartan geometries have to 
preserve the underlying structure given by the 
existence of the Cartan connection. Thus in the case of parabolic 
geometries, automorphisms especially preserve the induced filtration of the 
manifold. In the $|1|$--graded case, they correspond to the automorphisms 
of $G_0$--structures.
We define the symmetry on the parabolic geometry in the following way:
\begin{deff}
Let $(\ba \rightarrow M,\om)$ be a parabolic geometry. 
The (\emph{local}) \emph{symmetry with the center at $x \in M$} 
is a (locally defined) diffeomorphism $s_x$  
on $M$ such that:
\begin{itemize} 
\item[(i)] $s_x(x)=x$, 
\item[(ii)] $T_x s_x|_{T_x^{-1}M}=-\operatorname{id}_{T^{-1}_xM}$,
\item[(iii)] $s_x$ is covered by an automorphism $\va$ of the Cartan geometry, 
i.e. $s_x=\underline \va$ (on some neighborhood of $x$).
\end{itemize}
The geometry is called (\emph{locally}) \emph{symmetric} 
if there is a symmetry at each $x \in M$.
\end{deff}
In this paper, we will discuss only locally defined symmetries and locally 
symmetric geometries. We will omit the word `locally' and we will shortly say 
`symmetry' and `symmetric geometry'.  The relation between locally and 
globally defined symmetries we will discuss elsewhere. 

\smallskip
In other words, symmetries revert by the sign change only
the smallest subspace in the filtration, while their actions on the 
rest are completely determined by the algebraic structure 
of $\fg$~and~$\fp$.
To see that it is not reasonable to define the symmetry as an automorphism
such that its differential reverts the whole tangent space, it suffice to discuss 
the filtration and brackets on $|2|$--graded geometries, see \cite{ja}. See also 
\cite{KZ} for discussion of Cauchy--Riemann structures.

However, we are interested in $|1|$--graded parabolic geometries. In this case, 
the filtration is trivial and thus $T^{-1}M=TM$.  
First two properties in the definition than say that 
symmetries follow completely the classical intuitive idea. 
The third one gives that we can understand the latter defined symmetries 
as symmetries of the corresponding  $G_0$--structure
(and thus of the corresponding geometry). 

The class of $|1|$--graded geometries involves many well known types of 
geometries like conformal and projective structures. There are 
known generalizations of affine locally symmetric spaces to concrete examples of them, 
see e.g. \cite{P} for discussion of the projective case in the classical setup.
We use the universal language for parabolic geometries to study the properties 
of all symmetric $|1|$--graded geometries together. 
We start the discussion with one useful observation, see also \cite{ja}:
\begin{lemma} \label{prvekg} 
If there is a symmetry with the center at $x$ on a $|1|$--graded geometry of type 
$(G,P)$, then there exists an element $g \in P$ such that 
$
{\underline {\rm Ad}}_{g}(X)=-X 
$
for all $X \in \fg_{-1}$, where $\underline \Ad$ denotes the action 
on  $\fg_{-1}$ induced from the adjoint action.

All these elements are of the form $g=g_0\exp Z$, where 
$g_0 \in G_0$ such that $\Ad_{g_0}(X)=-X$
for all $X \in \fg_{-1}$ and $Z \in \fg_1$ is arbitrary. 
\end{lemma}
\begin{proof}
Let $\va$ cover some symmetry with the center at $x \in M$. The symmetry 
$\underline{\varphi}$ preserves the point 
$x$ and hence the morphism $\va$ preserves the fiber over $x$. Let $u$ be an
arbitrary fixed point in the fiber over $x$. There is an element $g \in P$
such that $\varphi(u) = u \cdot g$. We show that this $g$ satisfies the condition.

Let $\xi \in \mathfrak{X}(M)$ be a vector field on $M$. In the point 
$x$ we have
$$
T\underline{\varphi}.\xi(x)=-\operatorname{id}_{T_xM}(\xi(x))=-\xi(x)
.$$ 
We can also write $\xi(x)=\lz u, X \pz $ for a suitable $X \in \fg_{-1}$.
The symmetry simply changes the sign of the coordinates $X$ 
in the chosen frame~$u$
and we get
$$ 
T \underline \va(\lz u,X \pz ) = \lz u,-X \pz 
.$$ 
In the fiber over $x$, the equivariancy gives 
$$
T\underline \va (\lz u,X \pz ) =
\lz\varphi(u),X \pz =\lz ug,X \pz =\lz u, 
\underline{\text{Ad}}_{g^{-1}}(X) \pz 
.$$
Comparing the coordinates in the frame $u$ gives us the action of
element $g \in P$  on $\fg_{-1}$.  We have 
$\underline{\text{Ad}}_{g}(-X)=X$ and thus the action of element $g$ is 
the change of the sign for all elements from $\fg_{-1}$.

Next, because $g \in P$, we have $g=g_0\exp Z$ for some $g_0 \in
G_0$ and $Z \in \fg_1$. 
We have $\underline{\text{Ad}}_{g_0\exp Z}(X)=-X$ for all $X \in \fg_{-1}$. 
But the action of the component $\exp Z$ is trivial while the action of $g_0$
preserves the
grading, i.e. $\underline \Ad_{g_0}=\Ad_{g_0}$, and the element $g_0$ 
satisfies $\Ad_{g_0}(X)=-X$. 
\end{proof}
As a consequence of the proof of the Lemma \ref{prvekg} and 
Theorem \ref{auto} we get the following conditions for homogeneous models:
\begin{prop} \label{symm-homog}
(1)
All symmetries of the homogeneous model $(G \rightarrow G/P, \om_G)$ at 
the origin $o$ are exactly the 
left multiplications by elements $g \in P$
satisfying the condition in the Lemma \ref{prvekg}.
\\(2)
If there is a symmetry at the origin $o$, then the homogeneous model is
symmetric and there is an infinite amount of symmetries at each point.
\\(3)
If there is no such element, then none Cartan geometry of the same type carry 
some symmetry.
\end{prop}

The Proposition gives us nice and simple criterion to decide, whether the 
$|1|$--graded geometry of given type allow some symmetric geometry or not. 

\begin{exam} \emph{Projective structures.} \label{proj}
We can make two reasonable choices of the Lie group $G$ 
with the $|1|$--graded Lie algebra $\frak{sl}(m+1,\mathbb{R})$. We can consider
$G=Sl(m+1,\mathbb{R})$. Then the maximal $P$ 
is the subgroup of all matrices of the form
$\left( \begin{smallmatrix} d&W\\0& D\end{smallmatrix} \right)$ such that
${\frac1d}=\det D$ for $D \in Gl(m,\mathbb{R})$ and $W \in \mathbb{R}^{m*}$, 
but we take only the connected component of the unit. It consists of all 
elements such that $\det D>0$.

In this setting, the group $G$ acts on rays in $\mathbb{R}^{m+1}$ and $P$ 
is the stabilizer of the ray spanned by the first basis vector. 
Clearly, with this choice $G/P$ is diffeomorphic to the $m$--dimensional 
sphere. The subgroup
$G_0$ contains exactly elements of $P$ such that $W=0$, and this subgroup is
isomorphic to $Gl^+(m,\mathbb{R})$.

The second reasonable choice is $G=PGl(m+1,\mathbb{R})$, the quotient of
$Gl(m+1,\mathbb{R})$ by the subgroup of all multiples of the identity. 
This group acts on $\mathbb{R}P^m$ and as the subgroup $P$ we take the 
stabilizer of the line generated by the first basis vector. Clearly $G/P$
is diffeomorphic to $\mathbb{R}P^m$.
The subgroup $G_0$ is isomorphic to $Gl(m,\mathbb{R})$, because each class
in $G_0$ has exactly one representant of the form 
$\left( \begin{smallmatrix} 1&0\\0& D \end{smallmatrix} \right)$. 

We can make the
computation simultaneously and then discuss the cases separately.
We have 
$\fg=\left\{\left( \begin{smallmatrix}
-tr(A)&Z\\X& A\end{smallmatrix} \right)\ \Big|\ X,Z^T \in 
\mathbb{R}^m,\ A \in \frak{gl}(m,\mathbb{R}) \right\}$ 
and elements from $\fg_{-1}$ look like  $\left( \begin{smallmatrix}
0&0\\X& 0\end{smallmatrix} \right)$.
For each $a=\left( \begin{smallmatrix}
b&0\\0& B\end{smallmatrix} \right) \in G_0$ 
and $V=\left( \begin{smallmatrix}
0&0\\X& 0\end{smallmatrix} \right) \in \fg_{-1}$, 
the adjoint action $\operatorname{Ad}_aV$ is  given by $X \mapsto b^{-1}BX$.
We look for elements $\left( \begin{smallmatrix}
b&0\\0& B\end{smallmatrix} \right)$ such that $BX=-bX$ for each 
$X$. It is easy to see that $B$ is a diagonal matrix 
and that all elements on the diagonal are equal to $-b$. 
Thus we may represent the prospective solution as $\left( \begin{smallmatrix}
1&0\\0& -E\end{smallmatrix} \right)$.

Now we discuss the choice $G=Sl(m+1,\mathbb{R})$ with $G/P\simeq S^m$. 
The element has the
determinant $\pm 1$ and the sign depends on the dimension of the geometry.
If $m$ is even, then the element gives a symmetry but if $m$ is odd, 
then there is no symmetry on this model. 
In the case
of $G=PGl(m+1,\mathbb{R})$, the above element always represents 
the class in $G_0$
and thus yields the symmetry. In both cases, all elements giving a symmetry 
in the origin are of the form 
$\left( \begin{smallmatrix} 1&W\\0& -E\end{smallmatrix} \right)$ 
for all $W \in \mathbb{R}^{m*}$.

These two choices of the groups $G$ and $P$ correspond to 
projective structures on oriented and non--oriented manifolds. 
The non--oriented projective geometries can be symmetric, the homogeneous 
models are symmetric. The existence of a symmetry on
the oriented projective geometry depends on its dimension.
Clearly, the homogeneous model is the oriented sphere with the canonical 
projective structure (represented e.g. by the metric connection of the round
sphere metric) and the obvious symmetries are orientation preserving in the
even dimensions only.
Thus only the even--dimensional geometries can be symmetric. 
Symmetric odd--dimensional oriented
projective geometry does not exist. 
\end{exam}
There is the well known classification 
of semisimple Lie algebras in terms of simple roots and for a given $\fg$, there is a 
complete description of all parabolic subalgebras, see \cite{Y, parabook} for more 
details. This allows to describe all corresponding parabolic geometries (with 
simple Lie group), see \cite{WS,ja} for the list. One can see from the latter example 
that for concrete $|1|$--graded geometry, it depends mainly on the choice of groups 
$G$ and $P$ with the corresponding Lie algebras whether the homogeneous model of 
corresponding $|1|$--graded geometry is symmetric. 

For the geometries from the list, there are standard choices of groups which give 
well know geometries.  One can simply take and discuss this choices for each geometry 
from the list. We saw this discussion 
in the projective case (the case $G=Sl(m,\R)$) and one can also see Example 
\ref{kernel-conf} for the conformal case. In the other cases, the discussion is analogous.
 The other possibility is to ask, whether there is a choice 
of groups with the given $\fg$ and $\fp$ such that the corresponding 
model is symmetric. Again, we saw this approach in the latter Example 
(the case $G=PGl(m,\R)$). 

We remind  here only the best known examples, where is the existence of symmetries 
more or less clear. They are also the most interesting ones. See \cite{ja} for discussion 
of all $|1|$--graded cases. 
\begin{thm} \label{thm-hom}
Homogeneous models of the following $|1|$--graded parabolic geometries are symmetric:
\begin{itemize}
\item almost Grassmannian geometries modeled over the Grassmannians of 
$p$-planes in $\mathbb{R}^{p+q}$ (i.e. of type $(p,q)$) where $p,q \geq 2$, $G=PGl(p+q,\R)$,
\item projective geometries in dim $\geq  2$, $G=PGl(m+1,\R)$,
\item conformal geometries in all signatures in dim $\geq 3$, $G=O(p+1,q+1)$,
\item almost quaternionic geometries, $G=PGl(m+1,\mathbb{H})$.
\end{itemize}
\end{thm}
In other words, the latter geometries admit symmetric space -- the flat space~$G/P$.

\section{Torsion restrictions}
Motivated by the affine case, we find some restriction on the curvature of 
$|1|$--graded geometry carrying some symmetry. See \cite{KN2,H} for 
discussion of classical affine locally symmetric spaces in more detail.
We study the curvature of parabolic geometries in a similar way. 
We also give some corollaries of the results coming from the general theory 
of parabolic geometries.

\subsection*{Locally flat geometries.}
Remark that the curvature of a $|1|$--graded geometry is described by the curvature 
function $\ka: \ba \rightarrow \wedge^2\fg^{*}_{-1} \otimes \fg$ and the torsion 
is identified with the part $\ka_{-1}$. This is correctly defined component of the 
curvature, we have just to keep in mind the proper action of $P$ on $\fg_{-1}$.  

The following Proposition is the analogy of the classical result for the affine 
locally symmetric spaces, see also \cite{ja}. 
\begin{prop} \label{torsion-free}
Symmetric $|1|$--graded parabolic geometries are torsion free. 
\end{prop}
\begin{proof}
Choose an arbitrary $x \in M$ on a symmetric $|1|$--graded geometry of type 
$(G,P)$ and let $\va$ cover a symmetry with the center at $x$.  
The symmetry $\underline{\va}$ preserves $x$ and thus $\va$ preserves the 
fiber over $x$. The curvature function satisfies $\ka = \ka \circ \va$ and 
we have 
$$
\kappa(u) = \kappa(\varphi(u)) = \kappa(u \cdot g) = g^{-1} \cdot \kappa(u)
$$ 
for suitable $g \in P$ and the same holds for $\ka_{-1}$.
We compare $\ka_{-1}$ in the frames $u$ and $\va(u)$.  
Because $g$ is exactly the element from Lemma \ref{prvekg} which acts 
on $\fg_{-1}$ as $-$id,
we have 
\begin{align*}
\kappa_{-1}(\varphi(u))(X,Y)&=\kappa_{-1}(u \cdot g)(X,Y)=
g^{-1} \cdot \kappa_{-1}(u)(X,Y)=
\\ 
&= \underline \Ad _{g^{-1}}( \kappa_{-1}(u)(\underline 
\Ad _{g}X,\underline \Ad _{g}Y) )=
\\ 
&= -\kappa_{-1}(u)(-X,-Y)=-\kappa_{-1}(u)(X,Y).
\end{align*} 
This is equal to $\kappa_{-1}(u)(X,Y)$ and we obtain 
$\kappa_{-1}(u)(X,Y)= -\kappa_{-1}(u)(X,Y)$ for all $X,Y \in \mathfrak g_{-1}$.  
Thus $\kappa_{-1}(u)$
vanishes and this holds for all frames 
$u \in \mathcal{G}$ over $x$. The torsion then vanishes at $x$.

If the geometry is symmetric, then there is some symmetry at each $x \in M$. 
Then $\ka_{-1}$ vanishes for all $x \in M$ and the geometry is torsion free.
\end{proof}

The theory on harmonic curvature allows us to prove stronger restriction for 
many types of $|1|$--graded parabolic geometries. See \cite{Y,WS} for detailed 
discussion of the theory. 
As a corollary of Theorem \ref{thm-harmonic-curvature} and Proposition 
\ref{torsion-free} we get the following Proposition, see \cite{ja}: 
\begin{prop} \label{homog1}
Let $(\mathcal{G} \rightarrow M, \om)$ be a normal $|1|$--graded parabolic 
geometry such that its homogeneous components of
the harmonic curvature are only of degree $1$. If there is a  
symmetry at a point $x\in M$, then the whole curvature vanishes in this point. 
In particular, if the geometry is symmetric than it 
is locally isomorphic with the homogeneous model.
\end{prop}

The harmonic curvature can be defined only for regular normal geometries. 
It is much simpler object then the whole curvature and  the main feature is that 
it is possible to find algorithmically all its components. 
%%% and for most type of $|1|$--graded geometries, the essential parts have the 
%%% character of the torsion which has to vanish.

From the point of view of underlying structures, 
the condition `normal' is  assumption of technical character. 
For an appropriate $G_0$--structure on the manifold we can find  $|1|$--graded 
normal parabolic geometry inducing given structure. 
If we take symmetries as morphisms of the concrete underlying geometry, 
this assumption gives us no restriction.

\begin{rem} \label{homog3}
It is easy to see, that similar arguments apply in the case when the only 
homogeneous components are of degree $3$, see Theorem \ref{thm-harmonic-curvature}.
Clearly, the part $\ka_{1}$ vanishes for the same reason as the torsion.
\end{rem}

For all parabolic geometries, one can compute the corresponding components of 
the harmonic curvature. Computation of the cohomology of (complex) Lie algebras is based 
on Kostant's version of the Bott--Borel--Weil theorem and the algorithm can be found in
\cite{J,Josef}.
This allows us to list explicitely, which geometries satisfy the above condition on the 
harmonic curvature and which not. One gets that most of $|1|$--graded geometries has homogeneous components only of degree $1$ and has to be locally flat, if they are symmetric.
See \cite{ja} for discussion of all $|1|$--graded geometries. Only geometries from the list in 
Theorem \ref{thm-hom} are more interesting. Some of them satisfy conditions from the Proposition \ref{homog1} or Remark \ref{homog3}:

\begin{cor}
 Symmetric  normal $|1|$--graded geometries of the following types are locally flat:
\begin{itemize} 
\item almost Grassmannian geometries such that $p>2$ and $q>2$, 
\item conformal geometries in all signatures of dimension $3$, 
\item projective geometries of dimension $2$. 
\end{itemize}
\end{cor}
The other geometries from the Theorem \ref{thm-hom} are the only $|1|$--graded geometries 
which can carry symmetries in the points with nonzero curvature. 
\begin{thm} \label{interesting}
The following normal $|1|$--graded geometries can admit a symmetric space, 
which is not locally isomorphic to the homogeneous model of the same type:
\begin{itemize} 
\item projective geometries of dim $> 2$,
\item conformal geometries in all signatures of dim $> 3$, 
\item almost quaternionic geometries,
\item almost Grassmannian structures such that $p=2$ or $q=2$.
\end{itemize}
\end{thm}
The main property of this geometries is that they allow some homogeneous component 
of harmonic curvature of  degree $2$.
We have no restriction for these geometries coming from the theory of 
harmonic curvature.

\section{Weyl Structures and Symmetries}
We study the action of symmetries on so called Weyl structures, which 
provide more convenient understanding of the underlying structure on the manifold. 
See \cite{WS} for more detailed discussion on Weyl structures. We also 
discuss symmetries on effective geometries in a little more detail.

\subsection*{Weyl structures.}
Weyl structures are our main tool to deal with the interesting geometries.
They exist for all parabolic geometries.
We describe Weyl structures only in the $|1|$--graded case, 
general theory can be found in~\cite{WS, parabook}. 

Remind that there is the \emph{underlying bundle} $\ba_0:=\ba/\exp\fg_1$ for 
each $|1|$--graded geometry, which 
is the principal bundle $p_0:\ba_0 \rightarrow M$ with structure group $G_0$. 
At the same time we get the principal bundle $\pi: \ba \rightarrow \ba_0$ with 
structure group $P_+=\exp\fg_1$.

The \emph{Weyl structure} $\si$ for a $|1|$--graded geometry $(p:\ba \rightarrow M, \om)$ 
is a global smooth $G_0$--equivariant section of the projection 
$\pi:\ba \rightarrow \ba_0$. There exists some Weyl structure 
$\si: \ba_0 \rightarrow \ba$ for each $|1|$--graded geometry and for arbitrary two
Weyl structures $\si$ and $\hat \si$, there 
is exactly one $G_0$--equivariant mapping $\Upsilon:\ba_0\rightarrow \fg_1$ 
such that 
$$
\hat \si(u)= \si(u)\cdot \exp\Upsilon(u)
$$ 
for all $u \in \ba_0$. The equivariancy allows to extend $\U$ to $P$--equivariant 
mapping $\ba \rightarrow \fg_1$ and in fact, it is a $1$--form on $M$.
Weyl structures form an affine space modeled over the vector space of all 
$1$--forms and in this sense we can write $\hat \si = \si + \U$. 

The choice of the Weyl structure $\si$ induces the decomposition of all 
tractor bundles into $G_0$--invariant pieces. In particular, the adjoint 
tractor bundle splits as 
$$
\mathcal{A}M= TM \oplus {\rm End}_0(TM) \oplus T^*M
,$$
which is compatible with the latter facts on underlying structures. 
Thus ${\rm End}_0(TM)$ is the appropriate subbundle of ${\rm End}(TM)$ 
and the algebraic bracket of a vector field with a $1$--form becomes 
an endomorphism on $TM$. 
 
The choice of the Weyl structure $\si$ also defines the decomposition 
of the $1$--form 
$\si^*\om \in \Omega^1(\ba_0,\fg)$ such that 
$$
\si^*\om=\si^*\om_{-1}+\si^*\om_{0}+\si^*\om_{1}
.$$
The part $\si^*\om_0 
\in \Omega^1(\ba_0,\fg_0)$ defines the 
principal connection on $p_0:\ba_0 \rightarrow M$, the \emph{Weyl connection}.  
For each linear representation $V$
we get the induced Weyl connection $\na^\si$ on $\ba_0 \times_{\lambda} V$ and 
for arbitrary two Weyl structures $\si$ and $\hat \si = \si \cdot \exp \U$,
there is the explicit formula for the change of corresponding connections
$\na^\si$ and $\na^{\hat \si}$. For a vector field $\xi \in \frak X(M)$ and 
a section $s$ of $\ba_0 \times_{\lambda}V$ we have 
%%% It can be nicely written in the language of tractor calculi as
$$
\na^{\hat \si}_\xi(s)=\na^\si_\xi(s) + \{\xi,\U\} \bullet s
.$$
The algebraic bracket of a vector field with a $1$--form  becomes an endomorphism 
on $TM$ and $\bullet$ is just the algebraic action derived from $\lambda$.

Among general Weyl structures, there is an interesting class of them which are
crucial in the sequel -- the normal Weyl structures.
We define the \emph{normal Weyl structure} at $u$ as the only 
$G_0$--equivariant section $\si_u:\ba_0 \rightarrow  \ba$ 
satisfying 
$$
\si_u \circ \pi \circ {\rm Fl}_1^{\omi(X)}(u) = {\rm Fl}_1^{\omi(X)}(u)
.$$ 
Clearly, normal Weyl structures are defined locally over some neighborhood of $p(u)$ 
and can be used only for discussion of local properties. They are closely related 
to the normal coordinate systems for parabolic geometries and generalize the affine 
normal coordinate systems, see \cite{parabook,CSZ,V}.

Finally, let us shortly concentrate on automorphisms.
For each automorphism $\va$ of the geometry, there is the pullback 
$\va^*\si=\va^{-1}\circ \si \circ \va_0$ of the Weyl structure $\si$, 
where $\va_0$ is the \emph{underlying automorphism} on $\ba_0$ induced 
by $\va$. This is again some Weyl structure and there 
is exactly one $\U$ such that $\va^*\si = \si + \U.$ 
In addition, the pullback respects the affine structure, i.e. 
$$
\va^*(\si + \U)=\va^*\si + \va^*\U
.$$ 
There is also a crucial fact, that the pullback of normal Weyl structure is 
again normal Weyl structure. It is easy to see it from the following computation:
\begin{align*}
\va^*\si_u \circ \pi \circ {\rm Fl}_1^{\omi(X)}(u) 
&= \va^{-1} \circ \si_u \circ \va_0 \circ \pi \circ {\rm Fl}_1^{\omi(X)}(u)= 
\\& =\va^{-1} \circ \si_u \circ \pi \circ {\rm Fl}_1^{\omi(X)}(\va(u))=  
\\& = \va^{-1} \circ \va \circ  {\rm Fl}_1^{\omi(X)}(u) = {\rm Fl}_1^{\omi(X)}(u).
\end{align*}
Thus the pullback $\va^*\si_u$ again satisfies the conditions 
on normal Weyl structures.

\subsection*{Actions of symmetries on Weyl structures.}
Let us return to the $|1|$--graded geometries carrying some symmetries.
We discuss the action of coverings of symmetries on Weyl structures.

\begin{prop} \label{upsilon}
Let $(\ba \rightarrow M,\om)$ be a $|1|$--graded geometry and
suppose there is a symmetry 
with the center at $x \in M$
covered by some automorphism $\va$.
There is a Weyl structure $\si$ such that 
$$
\va^*\si \vert_{p^{-1}_0(x)}=\si \vert_{p^{-1}_0(x)}
.$$
Thus in the fiber over $x$,
the pullback of  this $\si$ along $\va$ equals to the 
same Weyl structure $\si$. 
\end{prop}
\begin{proof}
Choose an arbitrary Weyl structure $\hat \sigma$
and compute the pullback of this 
structure along $\varphi$. The result is another Weyl structure 
$\varphi^*\hat \sigma=\hat \sigma +\Upsilon.$ 
We prove that the Weyl structure $\hat \si + {1 \over 2} \U$ 
satisfies the condition.
We compute
\begin{align*}
\varphi^*(\hat \sigma + {1 \over 2}\Upsilon) 
&= \va^*\hat \si + \va^*{1 \over 2}\Upsilon= 
\hat \si + \U + {1 \over 2} \va^*\Upsilon 
\end{align*}
and it suffices to show that
$
\va^*\U(u)= (\U \circ \va_0)(u)=-\U(u)
$
holds for 
$u \in p_0^{-1}(x) \subset \ba_0$. 
Clearly, $\va_0:\ba_0\rightarrow \ba_0$ preserves $p_0^{-1}(x)$ and in fact, 
it is equal to the right action of some suitable element from $G_0$. 	
This is exactly the element $g_0$ from Lemma \ref{prvekg} corresponding 
to the frame $u$. Thanks to the equivariancy and the fact that 
the values of $\U$ are in $\fg_1$ (the dual of $\fg_{-1}$), the action of the 
element changes the sign and we get 
$(\U \circ \va_0)(u)=\U(ug_0)=-\U(u)$. 

At the point $x$ the latter fact gives 
$$
\hat \si + \U + {1 \over 2} \va^*\Upsilon 
= \hat \si +\U - {1 \over 2}\U= \hat \si+{1 \over 2}\U
.$$ 
If we put all together we get 
$\va^*(\hat \si+{1 \over 2}\U)=\hat \si+{1 \over 2}\U$ at $x$ and thus
the action of $\va$ preserves the Weyl structure 
$\si := \hat \si+{1 \over 2}\Upsilon$ in the fiber over $x \in M$. 
\end{proof}

In the sequel, we call any such Weyl structure \emph{invariant with 
respect to $\va$ at $x$} or shortly \emph{$\va$--invariant at $x$}.

There can be more than one invariant Weyl structure at $x$ with 
respect to the same $\va$, but all of them have to coincide at $x$. 
Let $\si, \bar \si$ be different $\va$--invariant Weyl structures at $x$.
We know that $\bar \si = \si + \U$ for some in general nonzero 
$\U:\ba_0 \rightarrow \fg_1$. At the point $x$ we get
$$
\bar \si=\va^*\bar \si= \va^*(\si + \U)=\va^*\si+\va^*\U=\si+\va^*\U
.$$ 
The relation $\si+\U=\si+ \va^* \U$ implies $\U=\va^*\U$ at $x$. 
Because $\va^*\U=-\U$ holds in the fiber over $x$ we get the 
vanishing of $\U$ at $x$.
In general, we know nothing about the neighborhood of $x$.
 
Remark that if $\si$ is $\va$--invariant Weyl structure at $x$ then 
$\va^*\si$ is again $\va$--invariant 
at $x$ because $\va^*\si=\si$ implies $\va^*(\va^*\si)=\va^*\si$ at $x$.
We can shortly say, that the pullback along $\va$ permutes all Weyl 
structures invariant with respect to $\va$ at $x$.
In addition, we can prove the following Theorem.
\begin{thm}
Suppose that $\va$ covers some symmetry with the center at $x$ 
on a $|1|$--graded geometry. There is exactly one normal Weyl 
structure $\si_u$ such that 
$$
\va^*\si_u=\si_u
$$ 
over some neighborhood of the center $x$.    
\end{thm}
\begin{proof}
Let $\si$ be an arbitrary $\va$--invariant Weyl structure at $x$, 
i.e. $\va^*\si=\si$ holds in the fiber over $x$. 
We take the normal Weyl structure $\si_u$ such that $\si_u(v) = \si(v)$ 
for $p_0(v)=x$. The condition of normality prescribes it then uniquely on a 
suitable neighborhood of $x \in M$.
%%%\marginpar{\scriptsize Prepsano, prekontrolovat...}
%%% \begin{alignat*}{2} 
%%% \si_u(v) &= \si(v) & &\quad\text{for} \ p_0(v)=x, \\ 
%%% \si_u(p \circ \text{Fl}^{\omi(X)}_1(\si(v))) 
%%% &= \text{Fl}^{\omi(X)}_1(\si(v)) & &\quad\text{otherwise}.
%%% \end{alignat*}
Pullback of this Weyl structure is again some normal Weyl structure. 
But we know that $\si$ and $\si_u$ coincide at $x$ and we have 
$\va^*\si_u=\si_u$ at $x$. Then $\va^*\si_u$ has to be the original normal
Weyl structure $\si_u$ and we get $\va^*\si_u=\si_u$ over some neighborhood of $x$.

Finally, the resulting normal Weyl structure does not depend on the choice of the 
$\va$--invariant Weyl structure $\si$ at $x$, because all these structures are equal 
at~$x$.
\end{proof}

\begin{cor}
Suppose there is a symmetry with the center at $x$ on a $|1|$--graded geometry. 
Then there is an admissible affine 
connection given on the neighborhood of $x$ which is invariant with 
respect to the symmetry.
\end{cor}
\begin{proof}
We take the Weyl connection $\na^{\si_u}$ given by the normal Weyl structure 
invariant on the neighborhood of $x$ with respect to some covering of the 
symmetry.
\end{proof}

\subsection*{Effective geometries} \label{effective}
Let us pass our discussion to the relation of various coverings of the 
symmetry at $x$ and their invariant Weyl structures at $x$. We also discuss here 
effective geometries in more detail.
We start with a
useful Lemma.

\begin{lemma} \label{id}
(1) The kernel $K$ of the parabolic geometry of type $(G,P)$
 is exactly the kernel of the adjoint 
action $\Ad: G \rightarrow Gl(\fg)$. In particular, it is contained in~$G_0$.  \\
(2) Let $\phi$ be an automorphism of a $|1|$--graded geometry 
such that its base morphism $\underline \phi$ preserves some $x \in M$. 
If $\phi(u)=u\cdot h$ for some $h \in K$  and for some $u \in p^{-1}(x)$ 
(thus for all $u \in p^{-1}(x)$), 
then  $\underline \phi={\rm id}_M$ on some neighborhood of $x$.
\end{lemma}
\begin{proof} 
(1) Remind that $K$ is the maximal normal subgroup of $G$ contained in $P$.
Suppose that $g$ belongs to the kernel $K$. We defined parabolic geometries as 
infinitesimally effective geometries and $K$ is discrete in this case.
Then thanks to the smoothness of the multiplication we have 
$ 
\exp(tX)g\exp(-tX)=g
$ 
for all $X \in \fg$ and the differentiating at $t=0$ gives 
$T\rho^{g}.X-T\lambda_{g}.X=0$.  Thus $\Ad_g(X)=X$ holds for all $X \in \fg$
and the element $g$ lies in the kernel of the adjoint action.

Suppose that $g$ belongs to the kernel of the adjoint action 
$\Ad: G \rightarrow Gl(\fg)$.
Then $\Ad_g:\fg\rightarrow \fg$ is identity or equivalently 
$\Ad_g(X)=X$ for all $X \in \fg$. 
If $h \in G$ is arbitrary then 
$\Ad_{hgh^{-1}}(X)=\Ad_h\Ad_g\Ad_{h^{-1}}(X)=\Ad_h(\Ad_{h^{-1}}(X))=X$ 
%%%\marginpar{\scriptsize Prepisovano, prekontrolovat...}
for all $X \in \fg$. The action of $g$ clearly respects the grading and 
$g$ belongs to $G_0 \subset P$. Thus elements from the kernel of the adjoint action 
form a normal subgroup of $G$ which is contained in $P$. 
Such subgroup has to be contained in $K$. \\
(2) We have $\underline \phi(x)=x$ and we use normal coordinates at 
$u \in p^{-1}(x)$ to describe the neighborhood of $x$. Any point from 
the suitable neighborhood of $x$ can be written as 
$p \circ {\rm Fl}_1^{\omi(X)}(u)$ for suitable $X \in \fg_{-1}$ and we have
\begin{align*}
\underline \phi \circ p \circ {\rm Fl}_1^{\omi(X)}(u)
&=p \circ \phi \circ{\rm Fl}_1^{\omi(X)}(u)= 
p \circ {\rm Fl}_1^{\omi(X)}(\phi(u))=\\&
= p \circ {\rm Fl}_1^{\omi(X)}(uh).
\end{align*}
The equivariancy of $\om$ and the fact that $\underline \Ad_h=\Ad_h$ for 
$h \in G_0$ give that the curve $p \circ {\rm Fl}_t^{\omi(X)}(uh)$ 
coincides with the curve $p \circ {\rm Fl}_t^{\omi(\Ad_{h^{-1}}X)}(u)$. 
(See \cite{CSZ,V} for details on generalized geodesics.) 
The action of $h$ from the kernel is trivial and then the action of $h^{-1}$ is 
trivial, too. We have 
$$
p \circ {\rm Fl}_1^{\omi(X)}(uh)
=p \circ {\rm Fl}_1^{\omi(\Ad_{h^{-1}}X)}(u)
=p \circ {\rm Fl}_1^{\omi(X)}(u)
.$$
This holds for all $X \in \fg_{-1}$ and $\underline \phi$ is the identity on a 
neighborhood of $x$.
\end{proof}

\begin{prop} \label{morf}
Let $\va$ and $\psi$ be two coverings of two symmetries with the center at $x$ 
on a $|1|$--graded geometry. \\
(1) Suppose that invariant Weyl structures with respect to this two coverings 
coincide at $x$, i.e. 
$\va^*\si|_{p_0^{-1}(x)} = \si|_{p_0^{-1}(x)}=\psi^*\si|_{p_0^{-1}(x)}$.
Then 
$ 
\psi(u) =\va(u)\cdot h 
$
holds for all $u$ over some neighborhood of $x$ and for some $h$ from the kernel 
of the geometry. In particular, $\va$ and $\psi$ have to cover the same symmetry 
with center at $x$.
\\ 
(2) Suppose that $\va$ and $\psi$ cover the same symmetry at $x$. Then  
Weyl structures invariant with respect to this two coverings coincide 
along the fiber over $x$.
\end{prop}
\begin{proof}
(1) Suppose that $\va^*\si=\si$ and $\psi^*\si=\si$ at $x$. Then we have 
$\va^{-1} \circ \si \circ \va_0=\psi^{-1} \circ \si \circ \psi_0$ at $x$ and 
this is equivalent with the fact that
$$
(\psi \circ \va^{-1}) \circ \si (v) = \si \circ (\psi_0 \circ \va_0^{-1})(v)
$$ 
holds for each $v \in p_0^{-1}(x)$. The morphism $\va_0$ preserves the 
fiber over $x$ and in fixed $v$ is equal to the right multiplication 
by some $k \in G_0$ such that $\Ad_{k}(X)=-X$ for all $X \in \fg_{-1}$. 
The morphism $\psi_0$ also coincides at $v$ with the action of some 
$g \in G_0$ satisfying $\Ad_{g}(X)=-X$ for all $X \in \fg_{-1}$. 
We can accordingly write 
$$
(\psi \circ \va^{-1}) \circ \si (v) = \si(v) \cdot h
,$$ 
where $h=k^{-1}g$ acts trivially on $\fg_{-1}$ by the $\Ad$--action.
Such element $h \in G_0$ has to act trivially on the whole $\fg$ because the action 
respects the grading and the action on $\fg_0 \subseteq  \fg_{-1}^* \otimes \fg_{-1}$ 
and $\fg_{1} \simeq \fg_{-1}^*$ is trivial. Thus $h$ belongs to the kernel and 
then
$
\underline{\psi \circ \va}^{-1} 
$ 
equals to the identity on some neighborhood of $x \in M$, see Lemma \ref{id}.
Then Theorem \ref{rigidity} gives that  
$$
\psi \circ \va^{-1}(u)={\rm id}_\ba \cdot f(u)$$ 
holds for some function $f:\ba \rightarrow K$ over the neighborhood of $x$. 
According to our definition 
of parabolic geometries, they are always infinitesimally effective and 
the function $f$ has to be constant on the neighborhood of $x$. 
The value of $f$  has to be the latter 
element $h \in G_0$.
We have $\psi \circ \va^{-1}(u)=uh$ and if we apply 
the automorphism $\va$ first, we get 
$\psi \circ \va^{-1}(\va(u))=\va(u)h$ over a suitable neighborhood of $x$.  
This implies $\psi(u)=\va(u) \cdot h$ over the neighborhood of $x$.  
\\(2) We have $\psi(u)=\va(u)\cdot h$  over some suitable neighborhood of $x$, 
where $h$ belongs to the kernel, see Theorem \ref{rigidity}. 
Then clearly their underlying automorphisms 
on $\ba_0$  satisfy $\psi_0(v)=\phi_0(v) \cdot h$ for $v \in \ba_0$. 
Suppose that $\va^*\si(v)= \va^{-1} \circ \si \circ \va_0(v) =\si(v)$ 
for all $v \in p_0^{-1}(x)$ and the Weyl structure $\si$. We have 
\begin{align*}
\psi^*\si(v)&= 
(\psi^{-1} \circ \si \circ \psi_0)(v)= 
\va^{-1} (\si \circ \psi_0(v))\cdot h^{-1}=\\
&=(\va^{-1} \circ \si \circ \va_0)(v) \cdot h \cdot h^{-1}=\va^*\si(v)=\si(v)
\end{align*}
for all $v \in p_0^{-1}(x)$ and the invariant Weyl structures coincide at $x$.
\end{proof}

In particular, if $\va$ covers some symmetry at $x$, then clearly $\va^{-1}$ covers 
some symmetry at $x$, too. Moreover, 
$\va$--invariant and $\va^{-1}$--invariant Weyl structures coincide at $x$.
Really, if we have $\va^*\si = \si$ at $x$ then we also have 
$
(\va^{-1})^* \va^*\si = (\va^{-1})^*\si
$
at $x$ and simultaneously we get 
$$
(\va^{-1})^* \va^*\si=(\va \circ \va^{-1})^*\si={\rm id}^*\si=\si
.$$ 
Thus we get $(\va^{-1})^*\si=\si$ at $x$. Then $\va$ and $\va^{-1}$ have to 
cover the same symmetry.
%%%\marginpar{\scriptsize Opravovano, prekontrolovat}
\begin{prop} \label{dusl}
Each symmetry on a $|1|$--graded geometry is involutive and the center of the 
symmetry is its isolated fixed point.
\end{prop}
\begin{proof}
The first part follows directly from the latter facts. The second part is 
obvious thanks to the fact that for each symmetry, its differential at the center  
acts as $-$id on the whole tangent space.
\end{proof}
If we start with an effective geometry, we in addition have following 
consequences of the latter Propositions: 
\begin{prop} 
Let $\va$ and $\psi$ cover two symmetries at $x$ on some effective 
$|1|$--graded geometry and suppose that their invariant Weyl structures 
coincide at $x$. Then $\va=\psi$ over a neighborhood of $x$. 
Moreover, each covering of a symmetry is involutive.
\end{prop}
\begin{proof}
From the proof of Proposition \ref{morf} we get that
$(\psi \circ \va^{-1}) \circ \si (v) = \si(v)$ 
holds in the case of effective geometries. The rest follows 
immediately from the formula.
\end{proof}

We can shortly summarize all these facts in the following way: 
\begin{itemize}
\item The symmetry can have several coverings, 
but all of them differ by the right multiplication by some element from the 
kernel and in the fiber over $x$, all of them share the same invariant Weyl 
structure at $x$. 
\item Each symmetry with the center at $x$ 
allows (in the fiber over $x$) exactly one invariant Weyl structure at $x$. 
This does not depend on the choice of the covering of the symmetry.
\item For effective geometries, each symmetry has exactly one covering.
\end{itemize}

Let us illustrate the situation with symmetries on (non)effective geometries on the 
homogeneous model of conformal geometry:
\begin{exam} {\it Conformal Riemannian structures.} \label{kernel-conf}
First, we have to show that homogeneous models of conformal structures are symmetric.
In Example \ref{exam-conf} we started with $G=O(p+1, q+1)$ and we would continue 
with this choice.

We are looking for elements giving symmetries in the origin of this model, 
see Corollary \ref{symm-homog}.
For
$b=\left( \begin{smallmatrix}
\lambda&0&0\\0&C&0\\0&0&\lambda^{-1} 
\end{smallmatrix} \right) \in G_0$ 
and
$V=\left( \begin{smallmatrix}
0&0&0\\X&0&0\\0&-X^TJ&0 
\end{smallmatrix} \right)
 \in \fg_{-1}$
the adjoint action Ad$_bV$ is given by $X\mapsto \lambda^{-1}CX$ 
and we require $\lambda^{-1}CX=-X$. 
Thus we look for $\lambda \in \mathbb{R} \setminus \{0\}$ and $C \in O(p,q)$ such that 
$CX=-\lambda X$ for each $X \in \mathbb{R}^{p+q}$. Clearly, $C$ has to be 
diagonal and all elements on the diagonal have to be equal to $1$ or $-1$ 
because $\det C$ is equal to $1$ or $-1$.
We get two elements $g_1=\left( \begin{smallmatrix}
-1&0&0\\0&E&0\\0&0&-1 \end{smallmatrix} \right)$ and $g_2=\left( \begin{smallmatrix}
1&0&0\\0&-E&0\\0&0&1 \end{smallmatrix} \right)$ satisfying all conditions.
Clearly, both elements belong to the group $O(p+1,q+1)$.
Then all elements inducing some symmetry in the origin are of the form 
$\Big( \begin{smallmatrix}
-1&-Z&{1 \over 2}ZJZ^T\\0&E&-JZ^T\\0&0&-1 \end{smallmatrix} \Big)$ and 
$ \left( \begin{smallmatrix}
1&Z&-{1 \over 2}ZJZ^T\\0&-E&JZ^T\\0&0&1 \end{smallmatrix} \right)$ 
for all $Z \in \mathbb{R}^{p+q*}$.

The choice $G=O(p+1,q+1)$ gives not an effective geometry. 
One can easily compute that 
the kernel consists of elements 
$\left( \begin{smallmatrix} 1&0&0\\0&E&0\\0&0&1 \end{smallmatrix} \right)$ and  
$\left( \begin{smallmatrix} -1&0&0\\0&-E&0\\0&0&-1 \end{smallmatrix} \right)$. 
Thus the elements $g_1$ and $g_2$ differ by 
multiplication by some element from the kernel. Left multiplication by elements 
from the kernel induces identity on the base manifold and then $g_1$ and $g_2$
give the same symmetry on the base manifold. Then each possible symmetry (in the origin) 
has exactly two possible coverings.

It is possible to take an effective model, e.g. to start with $PO(p+1,q+1)$, 
the factor of the orthogonal group by its center. With this choice we clearly find 
exactly one element from $G_0$ satisfying all conditions -- the class 
represented e.g. by the element 
$\left( \begin{smallmatrix} -1&0&0\\0&E&0\\0&0&-1 \end{smallmatrix} \right)$.
\end{exam}

\subsection*{Remark on relations to affine locally symmetric spaces}
(1) 
Remark first, that there is a notion on geodesic symmetry for affine 
locally symmetric spaces. Each symmetry (locally) reverses geodesics going 
through its center. This property describes the symmetry on some neighborhood 
of the center and can be used as a definition, see \cite{KN2,H}.
Analogies of the affine geodesics for parabolic geometries are generalized geodesics. 
They are defined as the projections of the flow lines of a horizontal vector 
field. 
Detailed discussion on generalized geodesics can be found in \cite{CSZ,V}. 

Our definition of the symmetry on $|1|$--graded geometries implies that symmetries are 
 those automorphisms, which revert the `classes' of generalized geodesics. 
More precisely, the symmetry at $x$ maps each generalized geodesic going through $x$ 
in some direction to some generalized geodesic going through $x$ in the opposite direction. 
This correspond to the fact that symmetries are not uniquely determined in this case.
We just saw this on homogeneous models, see Examples \ref{proj} and \ref{kernel-conf}.
There can exist a lot of different symmetries at one point on a $|1|$--graded geometry 
and it is not reasonable to define symmetries only via reverting of geodesics.

(2) On the affine (locally symmetric) space, there is exactly one normal coordinate system
at the point (up to $Gl(n,\R)$ transformation) and one can use it to describe nicely the symmetry at the point. In these 
coordinates the symmetry only reverts the straight lines going through the point.

On parabolic geometries, there can exist many different normal coordinate 
systems. They are given by the choice of the (second order) frame in the 
fiber. We showed that on $|1|$--graded geometry carrying some symmetry at $x$,
there is (up to some transformation given by an element from $G_0$) exactly one 
normal coordinate system $\si_u$ at $x$ such that the 
(covering of the) symmetry only reverts the straight lines going through the point 
in these coordinates.
Moreover, this fixes the connection $\na^\si$ which is compatible with the symmetry and therefore, its geodesics (which
are generalized geodesics because the affine connection 
is normal) are reversed by the symmetry.

(3) Remark, that there is an equivalent definition of the symmetry on the manifold. 
One can define symmetry at $x$ as a (locally defined) 
involutive automorphism such that the point $x$ is the isolated fixed point 
of this automorphism, see \cite{P} or \cite{KN2}.  Symmetries defined in this manner clearly satisfy the 
condition on the differential and they are symmetries in our sense. 
Conversely, the Proposition \ref{dusl} says that our symmetries correspond to 
the latter definition.

(4) Finally remind the well known description of affine locally symmetric spaces. 
The pair $(M,\na)$ is affine locally symmetric if and only if the torsion of $\na$ 
vanishes and its curvature is covariantly constant. For symmetric $|1|$--graded 
geometries, we have just proved that the torsion of the Cartan connection 
(and thus of all Weyl connections) vanishes. In the next section, we 
will discuss the curvature.

\section{Further curvature restriction}
In the Theorem \ref{interesting} we have indicated which geometries require some 
further study of the curvature and in this section, we deal only with them. 
We have to discuss so called Weyl curvature, which can be defined for all 
parabolic geometries via Weyl structures (and which is related to the 
homogeneous component of the curvature of homogeneous degree 2, see \cite{parabook}). 
It is an analogy of the classical object from conformal geometry, see \cite{WS} 
or \cite{TBI,TC} for more detailed discussion on objects which generalize the 
classical conformal ones. 

Finally we show some consequences for the concrete 
geometries. Let us first point out that the motivation for this results is the 
article \cite{P}, where the author studies the projective case in the classical 
setup of affine connections. Our methods work for all $|1|$--graded 
parabolic geometries.

\subsection*{Curvature restrictions}
From our point of view, the most interesting $|1|$--graded
geometries are the ones, which allow some homogeneous component 
of curvature of degree $2$. 
Suppose there is some symmetry at each point of such geometry. 
Then its curvature is of the form 
$\ka=\ka^0:\ba \rightarrow \wedge^2 \fg_{-1}^{*} \otimes \fg^{0}$ 
because symmetric $|1|$--graded geometries are torsion free.
If we choose some Weyl structure $\si$, we get the decomposition of 
$\si^*\ka=\ka \circ \si$ such that
$
\si^*\ka =\si^*\ka_0 + \si^* \ka_1.
$ 
The lowest homogeneity part of the decomposition  
$$
\si^*\ka_0: \ba_0 \rightarrow \wedge^2\fg_{-1}^{*} \otimes \fg_0
$$ 
does not change, if we change the Weyl structure. 
This part is called \emph{Weyl curvature} and is 
usually denoted by $W$.

If $\va$ is an automorphism of the parabolic geometry, 
then its curvature is invariant with respect to $\va$. If we choose a Weyl 
structure $\si$ such that $\va^*\si =\si$, then the Weyl curvature has to be 
invariant with respect to the underlying morphism. 
We study algebraic actions 
and covariant derivatives of the Weyl curvature $W$ with respect to Weyl 
connections. The existence of the invariant Weyl structure at the point is crucial 
for our considerations.

\begin{lemma} \label{nepevna}
Let $\si$ be an arbitrary Weyl structure on a $|1|$--graded geometry and 
let $\underline \va$ be a symmetry at $x$ covered by some $\va$. Then
$$
\{\xi ,\Upsilon \} \bullet  W + 2\nabla^{\sigma}_\xi W =0
$$
holds at $x$ for all $\xi \in \frak X(M)$, where $\U$ is defined by 
$\va^*\si=\si+\U$.
\end{lemma}
\begin{proof}
We take $\nabla ^\sigma_\xi  W(\eta ,\mu )$ for each 
$\xi,\eta,\mu \in \frak X(M)$ and we compute the pullback of the connection 
with respect to the symmetry $\underline \va$. 
At the point $x$ we have 
$$
(\underline{\varphi} ^*\nabla ^\sigma) _\xi  W(\eta ,\mu )=
(T\underline{\varphi}\otimes T\underline{\varphi}^{-1}). 
\nabla ^\sigma _{T\underline{\varphi}.\xi}
W(T\underline{\varphi}.\eta , T\underline{\varphi}.\mu)
=- \nabla ^\sigma_\xi  W(\eta ,\mu )
$$ 
for each $\xi,\eta,\mu \in \frak X(M)$. 
Next, we have $\va^*\si= \si+\U$ for some $\U$
and this gives 
$$
(\underline{\varphi} ^*\nabla ^{\sigma})_\xi 
 W(\eta ,\mu )= \nabla^{\sigma +\Upsilon}  _\xi  W(\eta ,\mu )
$$ 
for each $\xi,\eta, \mu \in \frak X(M)$. 
If we put all together, we get that the identity 
$$
-\nabla ^{\sigma}_\xi  W(\eta ,\mu )=\nabla^{\sigma +\Upsilon}_\xi  W(\eta ,\mu)
$$
holds at the point $x$.
Using the formula for change of Weyl connection we can rewrite this as  
$$
-\nabla^{\sigma}_\xi  W(\eta ,\mu ) = \nabla^{ \sigma}  _\xi  
W(\eta ,\mu )+ (\{ \xi ,\Upsilon \} \bullet W)(\eta ,\mu )
.$$ 
This holds for each $\xi,\eta,\mu \in \frak X(M)$ at the point $x$  and 
gives exactly the formula.
\end{proof}

As an easy consequence of the Lemma we get the following analogy of the result 
from affine locally symmetric spaces: 
\begin{thm} \label{pevna}
Suppose there is a symmetry with the center at $x$ on a $|1|$--graded geometry. 
Then there exists a Weyl connection $\na^\si$ such that $\na^\si W=0$ at the 
point~$x$. The connection corresponds to the invariant Weyl structure at $x$.
\end{thm}
\begin{proof}
Let $\va$ cover a symmetry $\underline{\va}$ with the center at $x$ and let $\si$ 
be the Weyl structure invariant with respect to $\va$ at $x$. 
Thus $\va^*\si=\si$ holds at the point $x$.
One can see all from the Lemma \ref{nepevna}. In this case, we have $\U=0$ at $x$ 
and the algebraic bracket from the expression in the  
Lemma has to vanish for all $\xi \in \frak X(M)$ at $x$. Then $2\na^\si_\xi W=0$ 
holds for all $\xi$ and this implies $\na^\si W=0$ at~$x$.
\end{proof}

In some sense, this is analogous to the classical results. 
On the affine locally symmetric space,
there is exactly one connection which is invariant with respect to all symmetries 
and we know that its curvature is covariantly constant with respect to the connection. 

In the case of symmetric $|1|$-graded geometries, there is the class of Weyl 
connections and we showed that at each point,
there is at least one of them such that the Weyl curvature is covariantly 
constant at the point.

\subsection*{Algebraic restrictions}
Using all latter facts we show some algebraic restriction on the Weyl curvature 
of symmetric $|1|$--graded geometries. Let us first introduce some conventions.

Let $\varphi $ and $\psi$ cover two different symmetries $\underline \va$ 
and $\underline \psi$ with the center at $x$ on a $|1|$--graded geometry. 
We showed above that in the fiber over $x$, the symmetries have different 
invariant Weyl structures at $x$.
Let $\si$ be $\va$-invariant Weyl structure at $x$, i.e. $\va^*\si=\si$ at $x$.
Then $\si$ cannot be $\psi$--invariant at $x$ and we have $\psi^*\si=\si+\U$ 
where $\U$ is nonzero at~$x$. 
In this way, we can find such $\Upsilon$ for each two coverings $\va$
and $\psi$ of two (different) symmetries. In fact, this $\U$ does not 
depend on the choice of the coverings of $\underline \va$ and $\underline \psi$ 
at the point $x$ because invariant Weyl structures of 
two different coverings of the same symmetry have to coincide along the fiber 
over $x$. In the sequel, we  call the form $\U$ at $x$ the
\emph{difference between $\underline \va$ and $\underline \psi$}. 
In the future, 
we often use the fact that this difference is nonzero at $x$, but we do not need the 
exact value. From this point of view, it is not important which 
symmetry is the first one and which is the second one.

\begin{prop} \label{omezeni}
Assume there are two different symmetries at the point $x$ on a $|1|$--graded geometry. 
Then 
\begin{align} \label{jedna}
\{\xi ,\Upsilon \} \bullet  W=0
\end{align} 
holds at $x$ for any $\xi \in \frak X(M)$, where $\U$ is the difference 
between the symmetries.

\end{prop}
\begin{proof}
Let $\underline \va$ and $\underline \psi$ be two different symmetries 
at $x$ with coverings $\varphi $ and $\psi$ and let $\U$ be the difference between 
$\underline \va$ and $\underline \psi$, i.e. $\va^*\si=\si$ and $\psi^*\si=\si+\U$ 
hold at the center $x$ for some Weyl structure $\si$ and $\U$ is nonzero at~$x$. 
The Lemma \ref{nepevna} and Theorem \ref{pevna} give that
$$
 \nabla^\sigma_\xi W=0
,$$  
$$
 \{\xi ,\U \} \bullet  W + 2\nabla^{\sigma}_\xi W = 0
$$ 
hold for all $\xi \in \frak X(M)$. If we put this together, we get the 
required expression.
\end{proof}

\begin{cor}
If there are two different symmetries with the center at $x$ on a $|1|$--graded 
geometry then 
\begin{align} \label{dva}
\begin{split}
\{\{\xi,\Upsilon\},W(\eta,\mu)(\nu)\}&-W(\{\{\xi,\U\},\eta\},\mu)(\nu)\ -
\\ W(\eta,\{\{\xi, \Upsilon \},\mu \})(\nu)&-W(\eta,\mu)(\{\{\xi ,\U \},\nu \})=0
\end{split}
\end{align}
holds at $x$ for all $\xi, \eta, \mu, \nu \in \frak X(M)$, where $\U$ is the difference
between the symmetries.
\end{cor}
\begin{proof}
The expression $\{\xi,\U\} \bullet W$ is of the type 
$\wedge^{2}T^*M \otimes T^*M \otimes TM$ for any vector field $\xi$ and we evaluate 
it on vector fields $\eta,\mu$ and $\nu$. We get:
\begin{align*}
(\{\xi ,\Upsilon \} \bullet  W)(\eta,\mu)(\nu)
=&\ \{\{\xi,\Upsilon\},W(\eta,\mu)(\nu)\}-W(\{\{\xi,\eta\},\eta\},\mu)(\nu)\ -\\
&\ W(\eta,\{\{\xi, \Upsilon \},\mu \})(\nu)-W(\eta,\mu)(\{\{\xi ,\U \},\nu \}).
\end{align*}
This implies directly the required formula.
\end{proof}

The latter Proposition gives some restriction on the curvature, which is not 
too clear. But we can use the formula $(1)$ to find some more clear restriction on the 
curvature for the interesting geometries. We should discuss the action of various 
elements, which can be written as the algebraic bracket of the difference and some 
vector field.
We would like to find some elements which act on the Weyl curvature in a sufficiently 
simple way. 
Then we could understand better the consequences of the restriction.

\begin{thm} \label{thm-proj} \label{thm-conf-indef}
Suppose there are two different symmetries at $x$ on a $|1|$--graded geometry.
Let $\U=\lz u,Z \pz$ be their difference at $x$ for some $Z \in \fg_{1}$ and 
$u$ from the fiber over $x$ and suppose there exists 
$X \in \fg_{-1}$ such that $[X,Z]$ acts by the adjoint action diagonalizable 
on $\fg_{-1}$ with eigenvalues $a_1, \dots, a_k$ such that 
$a_{i_1}+a_{i_2}+a_{i_3}-a_{i_4} \neq 0$ for arbitrary choice of them.
Then the Weyl curvature vanishes at $x$ and thus the whole curvature vanishes at $x$.
\end{thm}
\begin{proof}
Remind that $\U$ has to be nonzero at $x$ because the symmetries are different. 
For various vector fields $\xi$, we would like to discuss the action of the elements 
of the form $\{\U, \xi \}$ on the Weyl curvature, see Proposition \ref{omezeni}.
First possible simplification gives us the formula $(\ref{dva})$ which reduces the 
discussion of the action on Weyl curvature to the action on the tangent vectors.
We will discuss the case when the suitable elements act simply by multiplication by 
numbers. 

Let us now describe the situation using the definition of tractor bundles and 
reduce it to the computation in coordinates, see Section \ref{tract}. 
Let $u$ be the frame from the fiber over $x$ such that $\U(x)=\lz u, Z \pz$ 
for $Z \in \fg_{1}$.
For each vector field $\xi$ we can then write $\xi(x)=\lz u, X \pz$ for some 
$X \in \fg_{-1}$. In this way, we have $\{\xi,\U \}(x)=\lz u, [X,Z] \pz$ and so on. 

We would like to discuss the action of the latter bracket on vector fields at $x$. 
In coordinates at the frame $u$, we simply study the adjoint action of $[X,Z]$ on 
the whole $\fg_{-1}$. Let $X \in \fg_{-1}$ be as in the assumption. Thus $[X,Z]$ acts by the adjoint action diagonalizable on $\fg_{-1}$ with eigenvalues $a_1, \dots, a_k$ 
and $\fg_{-1}$ decomposes into the eingenspaces. 
%%% This property is invariant with respect to the change of the frame and 
Then the tangent space at $x$ decomposes in the same way by the action by 
$\{\xi, \U\}$ where $\xi(x)=\lz u, X \pz$. We have found $\xi \in \frak{X}(M)$ giving `understandable' 
action of the bracket.

Now, the formula $(\ref{dva})$ says how to understand the action of this element 
on the Weyl curvature. The Weyl curvature lives in the component of 
$\wedge^{2}T^*M \otimes T^*M \otimes TM$ and decomposes with respect to the 
decomposition of the tangent space. Thus up to the choice of the frame $u$, 
the action of the bracket on each component is multiplication by the sum of suitable 
eigenvalues. If this sum is always nonzero, then the Weyl curvature has to vanish. 

If there is some symmetry with the center at $x$, then the torsion vanishes 
at $x$. Existence of two different symmetries at
$x$ satisfying the condition kills the Weyl curvature at $x$ and then $\ka=\ka_1$. 
But $\ka_1$ has to vanish too, the reason is the same as in the case of the torsion.
\end{proof}

Now, we can simply verify the latter condition for each geometry separately. We will see later that the `simplification' given by the choice of concrete $\xi$ is sufficient to get quite strong  restrictions in concrete geometries. We will discuss conformal and projective structures in detail. The almost Grassmannian and almost quaternionic structures behave analogously and we give here only short remark rather than the precise description.

\subsection*{Conformal structures}
The Theorem \ref{thm-conf-indef} reduces the discussion of the Weyl curvature of 
geometries with more than one symmetry at the point to simple algebraic condition on 
Lie algebras 
of the corresponding geometry. We discuss here, whether the condition can be satisfied 
in the case of conformal geometry. The computations will be performed in the setting 
of Example \ref{exam-conf}.

\begin{thm}
Suppose there are two different symmetries with the center at $x$ on the conformal 
geometry of arbitrary signature and 
denote $\U$ their difference. Suppose that this $\U$ has nonzero length at $x$. Then the 
curvature $\ka$ vanishes at $x$.
\end{thm}
\begin{proof}
We will discuss the condition from the Theorem \ref{thm-conf-indef}.
Up to the choice of the frame, the difference can be represented by some matrix 
$Z=\left( \begin{smallmatrix} 0&V&0\\0&0&-JV^T\\0&0&0\end{smallmatrix} \right) \in \fg_{1}$ for
 some nonzero $V \in \R^{p+q}$. We choose
$X=\left( \begin{smallmatrix} 0&0&0\\JV^T&0&0\\0&-V&0\end{smallmatrix} \right) \in \fg_{1}$, 
just the dual of  $Z$. The Lie bracket of the elements is
$$
\left[\left( \begin{smallmatrix} 0&0&0\\JV^T&0&0\\0&-V&0\end{smallmatrix} \right),
\left( \begin{smallmatrix} 0&V&0\\0&0&-JV^T\\0&0&0\end{smallmatrix} \right) \right]= 
\left( \begin{smallmatrix} -VJV^T&0&0\\0&0&0\\0&0&VJV^T\end{smallmatrix} \right). 
$$
Here $VJV^T$ correspond to the square of the length of $\U$ and it is nonzero if and 
only if $\U$ has nonzero length. Under this condition, the latter element clearly acts
as a multiplication.  More precisely, we have    
$$
\Bigg[
\Bigg( \begin{smallmatrix} -VJV^T&0&0\\0&0&0\\0&0&VJV^T\end{smallmatrix} \Bigg),
\Bigg( \begin{smallmatrix} 0&0&0\\Y&0&0\\0&-JY^T&0\end{smallmatrix} \Bigg) 
\Bigg]= 
\Bigg( \begin{smallmatrix} 0&0&0\\VJV^TY&0&0\\0&-VJV^TJY^T&0\end{smallmatrix} \Bigg) 
.$$
Thus the whole $\fg_{-1}$ is one eigenspace and the condition on the 
sum of eigenvalues is always satisfied. Then the curvatures has to vanish, see Theorem 
\ref{thm-conf-indef}.
\end{proof}
\begin{rem}
One can also reduce some computation by the observation, that the bracket is just 
multiple of the grading element in the conformal case. Remind, that the grading element 
is the only element $E \in \fg_{0}$ with the property $[E,Y]=jY$ for all $Y \in \fg_j$.
It exists for each parabolic geometry and one can verify that 
in the conformal geometry, it is of the form  
$\left( \begin{smallmatrix} 1&0&0\\0&0&0\\0&0&-1\end{smallmatrix} \right)$. We just 
 multiply it by the number $-VJV^T=-|V|^2$.
\end{rem}
%%% As a simple consequence of the proof of the Theorem we get the following fact.
\begin{cor}
Suppose there are two different symmetries with the center at $x$ on the conformal geometry 
of positive definite or negative definite signature. Then the curvature vanishes at $x$.
\end{cor}
\begin{proof}
In these cases, the length of a nonzero vector is always nonzero. The rest follows immediately.
\end{proof}

\subsection*{Projective geometries}
Let us make now similar discussion for projective geometries. We will study the vanishing 
of the curvature at the point with more than one symmetry. 
We are again 
interested in the condition in Theorem \ref{thm-conf-indef}. We use the notation introduced 
in Example \ref{proj}.

\begin{thm} \label{thm-proj}
Suppose there are two different symmetries with the center at $x$ on a 
projective geometry. Then the curvature $\ka$ of the geometry vanishes at $x$.
\end{thm}
\begin{proof}
We will discuss the condition in the Theorem \ref{thm-conf-indef}. 
In contrast to the conformal case, we cannot find the bracket in the form 
of a multiple of the grading element. One can find the grading element and compute 
the bracket of $\fg_{-1}$ and $\fg_{1}$ in general. 
Up to the choice of the frame, we represent the difference by matrix 
$Z=\left( \begin{smallmatrix} 0&V\\0& 0\end{smallmatrix} \right)$ 
where $V=(1,0,\dots,0) \in \R^{m*}$. (We take concrete matrix, because in contrast to 
the conformal case, the computation in general is not transparent.)
We choose $X=\left( \begin{smallmatrix} 0&0\\V^T& 0\end{smallmatrix} \right)$.
The Lie bracket of the elements is
$$
\Big[ \Big( \begin{smallmatrix} 0&0\\V^T& 0\end{smallmatrix} \Big),
\Big( \begin{smallmatrix} 0&V\\0& 0\end{smallmatrix} \Big) \Big]=
\Big( \begin{smallmatrix} -VV^T&0\\0& V^TV\end{smallmatrix} \Big).
$$
For our choice, $VV^T=1$ and matrix $V^TV$ has $1$ on the first row and column and zero elswere.
The action of the element on $\fg_{-1}$ is following:
$$
\left[
\left( \begin{smallmatrix} -1&0&0&\dots&0\\0&1&0&\dots&0\\0&0&0&\dots&0 \\ \vdots& \vdots&\vdots&\ddots&\vdots\\0&0&0&\dots&0 \end{smallmatrix} \right),
\left( \begin{smallmatrix} 0&0&0&\dots&0\\a&0&0&\dots&0\\b&0&0&\dots&0 \\ \vdots &\vdots&\vdots&\ddots&\vdots\\c&0&0&\dots&0\end{smallmatrix} \right)
\right]= 
\left( \begin{smallmatrix} 0&0&0&\dots&0\\2a&0&0&\dots&0\\b&0&0&\dots&0 \\ \vdots &\vdots&\vdots&\ddots&\vdots\\c&0&0&\dots&0 \end{smallmatrix} \right)
.$$
Thus $\fg_{-1}$ decomposes into the eigenspaces for eigenvalues $1$ and $2$. The condition on the sum of eigenvalues is always satisfied and we have no other restriction.
Then the curvatures has to vanish, see Theorem \ref{thm-conf-indef}.
\end{proof}

\begin{cor}
For projective geometries, there can exist at most one symmetry at the 
point with nonzero curvature. 
If there are two different symmetries at each point, then the geometry is 
locally flat.
\end{cor}

\subsection*{Remark on further geometries}
We just shortly summarize analogous results for the remaining two geometries. 
\begin{thm}
(1) Suppose there are two different symmetries with the center at $x$ on the almost Grassmannian geometry of type $(2,q)$ or $(p,2)$ and denote $\U$ their difference. 
Suppose that this $\U$ has maximal rank at $x$. Then the curvature $\ka$ vanishes at $x$.
\\
(2) Suppose there are two different symmetries with the center at $x$ on an almost quaternionic
 geometry. Then the curvature $\ka$ of the geometry vanishes at $x$.
\end{thm}

The proofs are similar to the projective case and can be found in \cite{ja-dga07}. The more precise discussion and examples of the almost Grassmannian case can be also found in \cite{javojta}. 
\begin{rem}
The condition of maximal rank for the almost Grassmannian geometries of the latter type 
means that the rank of the difference is $2$ at $x$.
Remind that for the almost Grassmannian geometries, the other cases ($p>2$ and $q>2$) are not interesting, because each symmetric geometry of this type has to be locally flat.

Let us also note that in the lowest dimension ($p=q=2$), the almost Grassmannian geometry correspond to the conformal geometry of indefinite type and the condition on the length of the $\U$ agree with the condition on its rank.
\end{rem}

\end{document}